\RequirePackage{cmap} 
\documentclass[12pt]{amsart}
\usepackage[utf8]{inputenc}
\usepackage[T2A]{fontenc}
\usepackage{amsmath,amssymb,amsthm,blkarray,graphicx}
\usepackage[a4paper,hmargin=2.5cm,vmargin=2.5cm]{geometry}
\usepackage[english,russian]{babel}
\usepackage[mathscr]{euscript}
\usepackage{tikz}
\usepackage{tikz-cd}
\usetikzlibrary{arrows.meta,decorations.markings}
\usepackage[section]{placeins}

\tikzset{middlearrow/.style={
        decoration={markings,
            mark= at position 0.55 with {\arrow{#1}} ,
        },
        postaction={decorate}
    }
}
\definecolor{darkspringgreen}{rgb}{0.09, 0.45, 0.27}

\usepackage[pdftex,bookmarks=false,colorlinks=true,citecolor=darkspringgreen,debug=true,
  naturalnames=true,pdfnewwindow=true]{hyperref}

\newtheorem{thm}{Theorem}[subsection]

\newtheorem{lem}[thm]{Lemma}
\newtheorem{prop}[thm]{Proposition}

\theoremstyle{definition}
\newtheorem{defn}[thm]{Definition}

\theoremstyle{remark}
\newtheorem{rem}[thm]{Remark}

\newcommand{\nc}{\newcommand}
\nc{\on}{\operatorname} \nc{\wh}{\widehat}
\nc\ol{\overline} \nc\ul{\underline} \nc\wt{\widetilde}
\nc{\Hom}{\operatorname{Hom}} \nc{\reg}{\operatorname{reg}} \nc{\diag}{\operatorname{diag}}
\nc{\red}[1]{{\color{red}#1}}

\def\arxiv#1{\href{http://arxiv.org/abs/#1}{\tt arXiv:#1}} 

\numberwithin{equation}{subsection}

\begin{document}

\author{Ivan Motorin}
\address{Skolkovo Institute of Science and Technology}
\email{ivan.motorin@mail.ru}
\title[Resolution of the odd nilpotent cone of some Lie superalgebras]
      {Resolution of singularities of the odd nilpotent cone of orthosymplectic Lie superalgebras}

\selectlanguage{english}
\begin{abstract}
  We construct a Springer-type resolution of singularities of the odd nilpotent cone of the
  orthosymplectic Lie superalgebras $\mathfrak{osp}(m|2n)$.
\end{abstract}
\maketitle

\section{Introduction}

\subsection{}
It is impossible to list all the applications of the Grothendieck-Springer resolution of the
nilpotent cone of a semisimple Lie algebra. The theory of Lie superalgebras opens still wider
perspectives since there are many conjugacy classes of Borel subalgebras and hence many nonisomorphic
flag varieties. In the case of orthosymplectic Lie superalgebras $\mathfrak{osp}(m|2n)$,
C.~Gruson and S.~Leidwanger~\cite{GL} introduced the notion of {\em mixed} Borel subalgebras (having
the maximal possible number of odd simple roots) and studied the corresponding Grothendieck-Springer
resolutions of the odd nilpotent cone in the case $m=2n,2n+1$.

Their results were used in the computation of the finite dimensional irreducible characters of
the orthosymplectic Lie superalgebras~\cite{GL} and in the computation of the orthosymplectic
Kostka polynomials~\cite{BFT}. We extend the results of~\cite{GL} to all cases of $m,n$.
   
More precisely, we consider $\mathfrak{osp}(m|2n)$ algebra with associated vector spaces $V_0, V_1,\allowbreak \dim(V_0)=m, \dim(V_1)=2n$ equipped with orthogonal and symplectic forms respectively. We fix identifications $\alpha\colon  V_0 \rightarrow V_0^*, \beta\colon  V_1 \rightarrow V_1^*$ via the
corresponding forms. Now we can consider following elements of $\mathfrak{osp}(m|2n)$: let $A\in \Hom(V_0,V_1)$, then $A\oplus A^* \in \mathfrak{osp}(m|2n)$ where $A^* =\alpha \circ A^t \circ \beta^{-1}$, if $AA^* \in \mathfrak{sp}(V_1)$ is nilpotent (or equivalently $A^*A \in \mathfrak{so}(V_0)$ is nilpotent) then we call $A \oplus A^*$ nilpotent.\\
One can consider the variety of complete self-orthogonal flags in $V_0$ and $V_1$ (a complete self-orthogonal flag in $V_0,V_1$ is a maximal filtration such that $F_{0}^{(i)}\perp F_{0}^{(m-i)},F_{1}^{(i)}\perp F_{1}^{(2n-i)}$) and nilpotent operators $A$ respecting the flags
(in a sense explained in~Definition~\ref{defn1}), i.e.\ we set
\begin{equation}
    \Tilde{\mathcal{N}} = \{(A,F_0,F_1)\in (\mathcal{N}, \mathcal{B}_0',\mathcal{B}_1)| A\in \mathcal{N}_{F_0,F_1}\},
\end{equation}
where $\mathcal{B}_1,\mathcal{B}_0'$ denote the variety of complete self-orthogonal flags in $V_1$
and a connected component of the variety of such flags in $V_0$ respectively
(recall that the variety $\mathcal{B}_0$ of complete self-orthogonal flags in $V_0$
consists of~2 connected components if $m$ is
even), and $\mathcal{N}_{F_0,F_1}$ is the odd part of the mixed Borel subalgebra
of $\mathfrak{osp}(m|2n)$ compatible with the flags $F_0^{(\bullet)},F_1^{(\bullet)}$.
It has been proved in \cite{GL} that $pr_1\colon  \Tilde{\mathcal{N}} \rightarrow \mathcal{N}$ is a resolution of singularities for $m=2n,2n+1$. We extend the proof for $m=2n-1,2n+2$, and observe that
the same construction does not give a resolution of singularities for any other values of $m$. However, in other cases we construct a similar resolution making use of the partial
self-orthogonal flags. This is a particular case of the general construction of~\cite{HE}.

Also, in~\S\ref{sec3} we construct a Weierstra\ss\ section to the odd nilpotent cone in an
orthosymplectic Lie superalgebra. Additionally, we study the relation between the regularity property of
odd elements of an orthosymplectic Lie superalgebra and the regularity property of their
self-supercommutators. Finally, in~\S\ref{sec4} we apply the results of~\cite{DP} to find the
lower bound that guarantees the higher cohomology vanishing of dominant line bundles on our
resolutions.

\subsection{Acknowledgments}
This note arose from solution of exercises of Michael Finkelberg's course in Invariant Theory
https://math.hse.ru/Invariant\_Theory\_finkel2021spring. I am grateful to him for interesting
discussions. I am also indebted to the anonymous referee for very useful suggestions that allowed
to sharpen the main result. Many thanks to Vera Serganova for suggesting the idea of considering partial self-orthogonal flags for a resolution.\\
Section~\ref{sec4} was written during my stay in the Ben-Gurion University, I am grateful to Inna Entova for her help during this time period and to Alexander Élashvili for his remark on the work of Hesselink. \\
The author was partially supported by ISF 711/18 grant (PI Inna Entova-Aizenbud) and BSF grant 2019694 (PIs Inna Entova-Aizenbud and Vera Serganova)

    \section{Grothendieck-Springer resolution of the odd nilpotent cone}

    \subsection{Recollections from elementary geometry}
    For the reader's convenience we recall the following basic facts.
    
    \begin{lem}
      \label{lem1}
      The stabilizer of a given complete self-orthogonal flag $F_0$ or $F_1$ under the action of $SO(V_0)$ or $Sp(V_1)$ correspondingly is a Borel subgroup.
      \end{lem}

\begin{lem}
  \label{lem2}
  The variety $\mathcal{B}_0$ of complete self-orthogonal flags in $V_0$ has two connected components if m is even.
\end{lem}

\begin{rem} Note also that the flag $F$ for $SO(2n)$ is uniquely determined by its $F^{(i)},i\neq n$ components and the choice of connectedness component in the variety of self-orthogonal flags. This is simply because $F^{(n)}$ is obtained by adding a vector of the form $ae_n+be_{n+1}$ to $F^{(n-1)}$, but we have a condition that
\begin{equation}
    (ae_n+be_{n+1},ae_n+be_{n+1})=0,
\end{equation}
so either $a=0$ or $b=0$.
\end{rem}

\subsection{Some invariant theory}
It has been proved in~\cite{GL} that the odd nilpotent cone $\mathcal{N}$ of an orthosymplectic Lie
superalgebra $\mathfrak{osp}(m|2n)$ is irreducible. Also recall that Kraft and Procesi~\cite{KP} gave a full description of orthosymplectic orbits in the null cone via $ab$-diagrams and from this description we know that there are finitely many orbits. It means that there exists an open orbit. We can also tell exactly which orbit has maximal dimension since the dimension of the null cone can be inferred from the following lemma.\\
\begin{lem}
  \label{lem3}
  The maps
\begin{equation}
    q_0\colon  \Hom(V_0,V_1) \rightarrow \mathfrak{so}(V_0), A \mapsto A^* A; q_1\colon  \Hom(V_0,V_1) \rightarrow \mathfrak{sp}(V_1), A \mapsto AA^*
\end{equation}
realize isomorphisms $q_0\colon  \Hom(V_0,V_1)/\!\!/SO(V_0)\times Sp(V_1) \rightarrow \mathfrak{so}(V_0)/\!\!/SO(V_0) $ if $m\le 2n+1$ and $q_1\colon  \Hom(V_0,V_1)/\!\!/SO(V_0)\times Sp(V_1) \rightarrow \mathfrak{sp}(V_1)/\!\!/Sp(V_1) $ if $m\ge 2n+1$.
\end{lem}

\begin{proof} let $m\le 2n+1$ we can take an element of the following form $A=e_1 \otimes v_1 + \dots +e_m \otimes v_m$ where $e_i$ is orthonormal basis in $V_0$ and $v_i$ are some vectors from $V_1$ then it is easy to check that 
\begin{equation}
   A^tA= \sum_{1\le i,j \le m} \langle v_i,v_j \rangle E_{ij} 
\end{equation}
Any skew-symmetric matrix can be presented in a canonical form (with parameters $\lambda_1,\dots, \lambda_{\lfloor m/2 \rfloor}$) via the action of $SO(V_0)$. Because $m\le 2n+1$ we can always find such $v_i$ for a canonical form. Now we have to show the surjectivity of the map $q_0^*\colon  \mathbb{C}[\mathfrak{so}(V_0)^*] \rightarrow \mathbb{C}[V_0 \otimes V_1]^{Sp(V_1)}$. By the first fundamental theorem of Invariant Theory $\mathbb{C}[V_0 \otimes V_1]^{Sp(V_1)}$ is generated by quadratic functions of the form
\begin{equation}
    Q_{ij}(v_0 \otimes v_1, v_0' \otimes v_1') = \langle (e_i,v_0) v_1, (e_j,v_0')v_1'\rangle
\end{equation}
and $Q_{ij}= q_0^*(E_{ij}^{\vee})$.\\
Now let $m\ge 2n+1$ we can take an element of the following form $A= v_1\otimes e_1 + \dots+ v_{2n}\otimes e_{2n}$ where $e_i$ is a basis of $V_1$ in a canonical form and $v_i$ are some elements from $V_0$. It is easy to check that
\begin{equation}
    AA^t = \sum_{1\le i,j \le 2n} (v_i,v_j) E_{ij} \cdot J
\end{equation}
where $J$ is the symplectic form. If we choose $v_i=e_i + ie_{i+n}, v_{n+i}=\lambda_i (e_i-ie_{i+n})$ then we will get any element from the Cartan subalgebra, so $q_1$ is dominant. By the first fundamental theorem of Invariant Theory for $SO(V_0)$ invariants of $\mathbb{C}[V_0 \otimes V_1]$ are generated by analogous quadratic functions and the determinant, but since $\dim(V_1) < \dim(V_0)$ the determinant is identically zero and thus $\mathbb{C}[V_0 \otimes V_1]^{SO(V_0)}=\mathbb{C}[V_0 \otimes V_1]^{O(V_0)}$. By repeating the argument as for $Sp(V_1)$ we get the statement of the lemma.
\end{proof}

\subsection{Resolution of the odd nilpotent cone}
\begin{defn}
  \label{defn1}
  An operator $A$ respects complete self-orthogonal flags iff:\\
a) $AF_0^{(i)}\subset F_1^{(i-1)}, 1\le i\le 2n+1$ in the case $m=2n+1$.\\
b) $AF_0^{(i)}\subset F_1^{(i-1)}, 1\le i\le n, AF_0^{(i+1)}\subset F_1^{(i-1)},n+1\le i\le 2n+2$ in the case $m=2n+2$.\\
c) $AF_0^{(i)}\subset F_1^{(i)}, 1\le i\le n, AF_0^{(i)}\subset F_1^{(i-1)},n+1\le i\le 2n$ in the case $m=2n$.\\
d) $AF_0^{(i)}\subset F_1^{(i)}, 1\le i\le 2n-1$ in the case $m=2n-1$.\\
Note that since $\langle Av,w \rangle=(v,A^* w), v\in V_0, w\in V_1$ (which can be checked on the basis vectors of $V_0$ and $V_1$) these conditions are equivalent to the following:\\
a) $A^* F_1^{(2n+1-i)} \subset F_0^{2n+1-i}$ in the case $m=2n+1$.\\
b) $A^*F_1^{(2n+1-i)}\subset F_0^{(2n+2-i)} (i\le n), A^*F_1^{(2n+1-i)}\subset F_0^{(2n+1-i)} (i>n)$ in the case $m=2n+2$.\\
c) $A^*F_1^{(2n-i)}\subset F_0^{(2n-i)} (i\le n),A^* F_1^{(2n+1-i)}\subset F_0^{(2n-i)} (i>n)$ in the case $m=2n$.\\
d) $A^*F_1^{(2n-1-i)}\subset F_0^{(2n-i)}$ in the case $m=2n-1$.
\end{defn}

Thus we see that the conditions in Definition~\ref{defn1} imply the nilpotency of $A$. Additionally, one can compare this definition with~\cite[Definition 5]{GL} and see that the element $(A,A^*)\in \Hom(V_0,V_1)\times \Hom(V_1,V_0)$ belongs to the odd part of the corresponding mixed Borel subalgebra of $\mathfrak{osp}(m|2n)$.\\
Now let us consider the cases above.
\begin{prop}
  \label{lem4}
  The map $pr_1$ is surjective.
\end{prop}
This proposition
is a consequence of~\cite[Proposition 8]{GL}. However, we give our own explicit proof for $m=2n-1,2n,2n+1,2n+2$ here.
\begin{proof} We will prove this in three cases.\\
\textbf{a)} Let $m=2n+1$. We know from the classification of nilpotent operators by Kraft and Procesi~\cite{KP} and A.~Berezhnoy~\cite{B} that $A$ decomposes into direct sum of blocks of type $\alpha_k,\beta_k,\gamma_k,\delta_k,\epsilon_k$ in Kraft's notation, besides dual vectors are paired in an antidiagonal manner (like on the diagrams below, bigger cases are analogous to the examples):
\begin{equation}
\begin{tikzcd}
A\colon a\rightarrow b                                           & A^*\colon  b\rightarrow a                                      & \alpha_1\colon                                                                   &               &               &   \\
a_1 \arrow[r] \arrow[rrrr, no head, dashed, bend right=49] & b_1 \arrow[r] \arrow[rr, no head, dashed, bend right=49] & a_2 \arrow[r] \arrow[no head, dashed, loop, distance=2em, in=305, out=235] & b_2 \arrow[r] & a_3 \arrow[r] & 0
\end{tikzcd}
\end{equation}
\begin{equation}
    \begin{tikzcd}
\beta_1\colon  & b_1 \arrow[r] \arrow[rr, no head, dashed, bend right=49] & a_1 \arrow[r] \arrow[no head, dashed, loop, distance=2em, in=305, out=235] & b_2 \arrow[r] & 0
\end{tikzcd}
\end{equation}
\begin{equation}
    \begin{tikzcd}
\epsilon_1\colon  & a_1 \arrow[r] \arrow[rd, no head, dashed] & b_1 \arrow[r] \arrow[ld, no head, dashed] & 0 \\
            & b_2 \arrow[r]                             & a_2 \arrow[r]                             & 0
\end{tikzcd}
\end{equation}
\begin{equation}
    \begin{tikzcd}
\gamma_1\colon  & a_1 \arrow[r] \arrow[rrd, no head, dashed] & b_1 \arrow[r] \arrow[d, no head, dashed] & a_2 \arrow[r] & 0 \\
          & a_3 \arrow[r] \arrow[rru, no head, dashed] & b_2 \arrow[r]                            & a_4 \arrow[r] & 0
\end{tikzcd}
\end{equation}
\begin{equation}
\begin{tikzcd}
\delta_2\colon  & b_1 \arrow[r] \arrow[rrrrd, no head, dashed] & a_1 \arrow[r] \arrow[rrd, no head, dashed] & b_2 \arrow[r] \arrow[d, no head, dashed] & a_2 \arrow[r] & b_3 \arrow[r] & 0 \\
          & b_4 \arrow[r] \arrow[rrrru, no head, dashed] & a_3 \arrow[r] \arrow[rru, no head, dashed] & b_5 \arrow[r]                            & a_4 \arrow[r] & b_6 \arrow[r] & 0
\end{tikzcd}
\end{equation}
In other words $\gamma_k,\delta_k,\epsilon_k$ rows span Lagrangian subspaces and dual vectors are obtained by the central symmetry of a row. Now because $m=2n+1$ we can decompose any $a,b$-diagram of a nilpotent $A$ into blocks $\epsilon$, pairs of blocks $\gamma$ and $\delta$, block $\alpha$ and block $\beta$, block $\gamma$ and two blocks of $\beta$, block $\delta$ and two blocks of $\alpha$, so that there would be a remainder of the form $\alpha$ or of the form $\gamma$ and $\beta$. It is necessary and sufficient to find two flags $F_1$ and $F_0$ which will be respected by $A$ as in the definition 1. Let us pick a specific basis in each case described above for a convenient flag description. If we have a block $\gamma_{2k+1}$ and $\delta_{2r}$ we will denote vectors in the upper row of $\gamma$ by increasing index $a_1 \rightarrow b_1 \dots  a_{2k+2}\rightarrow 0$ and vectors in the lower row will be denoted in the similar way (as on the diagram above) $a_{2k+3}\rightarrow \dots$. We will label vectors of $\delta$ in the same way, but with apostrophes.\\
Let us introduce the following vectors:
\begin{equation}
    e_{1} = a_{2k+2},e_{2}=a_{2k+1},\dots,e_{2k+2}=a_1, e_{2k+3}= a_{2r}',e_{2k+4}= a_{2r-1}',\dots,e_{2k+2r+2}= a_{1}' 
\end{equation}
\begin{multline}
    e_{2k+2r+3} = a_{4r}',e_{2k+2r+4}=a_{4r-1},\dots,e_{2k+4r+2}=a_{2r+1}', e_{2k+4r+3}= a_{4k+4},\\
    e_{2k+4r+3}= a_{4k+3},\dots,e_{4k+4r+4}= a_{2k+3}
\end{multline}
and $b$-basis
\begin{multline}
    f_{1}=b_{2k+1},\dots,f_{2k+1}=b_{1},f_{2k+2}=b_{2r+1}',f_{2k+2r+2}=b_{1}',f_{2k+2r+3}=b_{4r+2}',\dots,\\
    f_{2k+4r+3}=b_{2r+2}',f_{2k+4r+4}=b_{4k+2},\dots,f_{4k+4r+4}=b_{2k+2}
\end{multline}
In the case of $\epsilon_k$ we will introduce analogous basis
\begin{equation}
    e_{i}=a_{2k+1-i},f_{i}=b_{2k+1-i}, 1\le i\le 2k
\end{equation}
In other cases the choice of basis is a bit harder because the first half of them must be self-orthogonal. Suppose that we have $\alpha_k,\beta_m$ and $k>0$, in each block $a,b$-basis is enumerated as in the previous cases. We have to start the $a$-basis from the kernel of $A$ in $\alpha_{k}$. However, depending on sizes of blocks $k$ and $r$ we use different technique for the choice. We will demonstrate the procedure on the examples below, which can be easily generalized.\\
1) $a$-chain is longer than $b$-chain:
\begin{equation}
\begin{gathered}
    a_1 \rightarrow b_1 \rightarrow a_2 \rightarrow b_2 \rightarrow a_3 \rightarrow 0\\
    b_1' \rightarrow a_1' \rightarrow b_2' \rightarrow 0\\
    e_{1}=a_3,e_{2}=a_2+ia_1',e_{3}=a_2-ia_1',e_{4}=a_1,\\
    f_{1}=b_2+ib_2',f_{2}=b_2-ib_2',f_{3}=b_1+ib_1',f_{4}=b_1-ib_1'
\end{gathered}
\end{equation}
2) $b$-chain is longer than $a$-chain:
\begin{equation}
\begin{gathered}
    a_1 \rightarrow b_1 \rightarrow a_2 \rightarrow b_2 \rightarrow a_3 \rightarrow 0\\
    b_1' \rightarrow a_1' \rightarrow b_2' \rightarrow a_2' \rightarrow b_3' \rightarrow a_3' \rightarrow b_4' \rightarrow 0\\
    e_{1}=a_3,e_{2}=a_3',e_{3}=a_2+ia_2',e_{4}=a_2-ia_2',e_{5}=a_1',e_{6}=a_1,\\
    f_{1}=b_4',f_{2}=b_2+ib_3',f_{3}=b_2-ib_3',f_{4}=b_1+ib_2',f_{5}=b_1-ib_2',f_{6}=b_1'
\end{gathered}
\end{equation}
The $\alpha_0$ blocks have to receive special treatment, because the are not self-orthogonal. If there are an even number of them, we can count two $\alpha_0$ as an exceptional block $\alpha_0'$,  make dual vectors $a_1+ia_1' \rightarrow 0,a_1-ia_1' \rightarrow 0$ and work with them. If there is an odd nuber of them, we will leave one $\alpha_0$ as a remainder for the whole partition (we will put it in the middle of the flag $F_0$ later). It is easy to see that only two $\beta$ or $\delta$ can be paired with $\alpha_0'$. We make the following choice of basis in the first case:
\begin{equation}
\begin{gathered}
    b_1 \rightarrow a_1 \rightarrow b_2 \rightarrow 0\\
    b_3 \rightarrow a_2 \rightarrow b_4 \rightarrow 0\\
    a_1' \rightarrow  0\\
    a_1'' \rightarrow 0\\
    e_{1}=a_1'+ia_1'',e_{2}=a_1+ia_2,e_{3}=a_1-ia_2,e_{4}=a_1'-ia_1'',\\
    f_{1}=b_2+ib_4,f_{2}=b_2-ib_4,f_{3}=b_1+ib_3,f_{4}=b_1-ib_3
\end{gathered}
\end{equation}
In the second case we take basis analogous to the previous cases. If we have $\gamma$ and two $\beta$, we as always start the construction of the basis from the Lagrangian part of the $\gamma$ then move to the two $\beta$ and use combinations of the form $a_i\pm ia_j',b_i\pm b_j'$. The case of $\delta$ and two $\alpha$ is easier, as we just start from the end of the first $\alpha$, go to the left until the middle vector, then we repeat the same thing with the second $\alpha$, take the combination of the middle vectors $a_i \pm a_j'$ (look at the example of two $\alpha$ for $m=2n+2$) and use a row from $\delta$.\\
Now if there is an $\alpha_k$ in the remainder, we write $e_{i}=a_{2k+2-i},f_{i}=a_{2k+1-i}$. And if we have $\gamma_k, \beta_m$, we start from the row of $\gamma_k$ from $a$, go to the $\beta_m$, continue constructing the basis from the end of the chain to the beginning and finish with the dual part of $\gamma_k$ to its first half (that way the central $a$-element in $\beta_m$ will be situated in the center of $a$-basis).\\
The flag is constructed like follows: we define new vectors $e_i,f_j$ (as in \cite{GL}) by putting vectors of the remainder into the middle of two flags, then we take any block described above, add first half of $a$-vectors to the left of already put vectors and the second half to the right, we repeat the same process with $b$-vectors and then move to the next block. It is easy to check that $A$ respects such two flags. Here is an example,
\begin{equation}
\begin{gathered}
\epsilon_1 +\alpha_0\colon 
    a_1 \rightarrow b_1 \rightarrow 0\\
    b_2 \rightarrow a_2 \rightarrow 0\\
    a_1' \rightarrow  0\\
    F_0\colon  \{0\}\subset \langle a_2 \rangle\subset \langle a_1',a_2 \rangle\subset \langle a_1,a_2,a_1'\rangle\\
    f_1\colon  \{0\}\subset \langle b_2\rangle \subset \langle b_1,b_2\rangle
\end{gathered}
\end{equation}
\textbf{b)} Now let $m=2n+2$ then we can have $\gamma$ or two $\alpha$ as a remainder. In the first case we construct the middle of the flag (again here we show how to construct $e_i',f_j$ basis from \cite{GL}) as follows
\begin{equation}
\begin{gathered}
    a_1 \rightarrow b_1 \rightarrow a_2 \rightarrow 0\\
    a_3 \rightarrow b_2 \rightarrow a_4 \rightarrow 0\\
    e_{1}=a_2,e_{2}=a_1, e_{3}=a_4, e_{4}=a_3, f_{1}=b_1,f_{2}=b_2 \text{ or }\\
    e_{1}=a_2,e_{2}=a_4, e_{3}=a_1, e_{4}=a_3, f_{1}=b_1,f_{2}=b_2
\end{gathered}
\end{equation}
Different choice of basis is responsible for the choice of connectedness component for $F_0$. As in the case of $m=2n+1$ because of the shift at the center the condition $AF_0^{(i+1)}\subset F_1^{(i-1)}, (i>n)$ also holds if we make the same choice of $a,b$-basis. If we have two $\alpha$ in the remainder.
\begin{equation}
\begin{gathered}
    a_1 \rightarrow b_1 \rightarrow a_2 \rightarrow b_2 \rightarrow a_3 \rightarrow 0\\
    a_1' \rightarrow b_1' \rightarrow a_2' \rightarrow b_2' \rightarrow a_3' \rightarrow 0\\
    e_{1}=a_3,e_{2}=a_3', e_{3}=a_2+ia_2', e_{4}=a_2-ia_2',e_{5}=a_1',e_{6}=a_1, f_{1}=b_2,f_{2}=b_2',f_{3}=b_1',\\ f_{4}=b_1 \text{ or }\\
    e_{1}=a_3,e_{2}=a_3', e_{3}=a_2-ia_2', e_{4}=a_2+ia_2',e_{5}=a_1',e_{6}=a_1, f_{1}=b_2,f_{2}=b_2',f_{3}=b_1',\\ f_{4}=b_1
\end{gathered}
\end{equation}
\textbf{c)} Finally, suppose $m=2n$. Every partition can be divided into the following blocks: $\epsilon$, $\alpha$ and $\beta$, $\gamma$ and $\delta$, $\gamma$ and two $\beta$, $\delta$ and two $\alpha$. This time we have to choose new basis for these blocks (we will again demonstrate this procedure by examples). If we have a block $\epsilon$:
\begin{equation}
\begin{gathered}
    a_1 \rightarrow b_1 \rightarrow a_2 \rightarrow b_2 \rightarrow 0\\
    b_3 \rightarrow a_3 \rightarrow b_4 \rightarrow a_4 \rightarrow  0\\
    e_{1}=a_2,e_{2}=a_1, e_{3}=a_4, e_{4}=a_3, f_{1}=b_2,f_{2}=b_1,f_{3}=b_4, f_{4}=b_3 \text{ or }\\
    e_{1}=a_2,e_{2}=a_4, e_{3}=a_1, e_{4}=a_3, f_{1}=b_2,f_{2}=b_1,f_{3}=b_4, f_{4}=b_3
\end{gathered}
\end{equation}
Suppose that we have $\alpha$ and $\beta$ - a case with $\alpha_0$ is not special this time, but lengths of the chains may be different
\begin{equation}
\begin{gathered}
    a_1 \rightarrow b_1 \rightarrow a_2 \rightarrow b_2 \rightarrow a_3 \rightarrow 0\\
    b_1' \rightarrow a_1' \rightarrow b_2' \rightarrow 0\\
    e_{1}=a_3,e_{2}=a_2+ia_1', e_{3}=a_2-ia_1', e_{4}=a_1, f_{1}=b_2',f_{2}=b_2,f_{3}=b_1, f_{4}=b_1' \text{ or }\\
    e_{1}=a_3,e_{2}=a_2-ia_1', e_{3}=a_2+ia_1', e_{4}=a_1, f_{1}=b_2',f_{2}=b_2,f_{3}=b_1, f_{4}=b_1'
\end{gathered}
\end{equation}
or
\begin{equation}
\begin{gathered}
    a_1 \rightarrow b_1 \rightarrow a_2 \rightarrow b_2 \rightarrow a_3 \rightarrow 0\\
    b_1' \rightarrow a_1' \rightarrow b_2' \rightarrow a_2' \rightarrow b_3' \rightarrow a_3' \rightarrow b_4' \rightarrow 0\\
    e_{1}=a_3+ia_3',e_{2}=a_3-ia_3', e_{3}=a_2+ia_2', e_{4}=a_2-ia_2',e_{5}=a_1+ia_1',e_{6}=a_1-ia_1', f_{1}=b_4',\\ f_{2}=b_2 + ib_3',f_{3}=b_2 -ib_3, f_{4}=b_1+ib_2',f_{5}=b_1-ib_2',f_{6}=b_1' \text{ or }\\
    e_{1}=a_3+ia_3',e_{2}=a_3-ia_3', e_{3}=a_2-ia_2', e_{4}=a_2+ia_2',e_{5}=a_1+ia_1',e_{6}=a_1-ia_1', f_{1}=b_4',\\ f_{2}=b_2 + ib_3',f_{3}=b_2 -ib_3, f_{4}=b_1+ib_2',f_{5}=b_1-ib_2',f_{6}=b_1'
\end{gathered}
\end{equation}
If we have $\gamma$ and $\delta$
\begin{equation}
\begin{gathered}
    a_1 \rightarrow b_1 \rightarrow a_2 \rightarrow 0\\
    a_3 \rightarrow b_2 \rightarrow a_4 \rightarrow 0\\
    b_1' \rightarrow a_1' \rightarrow b_2' \rightarrow a_2' \rightarrow b_3' \rightarrow 0\\
    b_4' \rightarrow a_3' \rightarrow b_5' \rightarrow a_4' \rightarrow b_6' \rightarrow 0\\
    e_{1}=a_2',e_{2}=a_1', e_{3}=a_2, e_{4}=a_1,e_{5}=a_4,e_{6}=a_3,e_{7}=a_4',e_{8}=a_3', f_{1}=b_3',\\ f_{2}=b_2', f_{3}=b_1', f_{4}=b_1,f_{5}=b_2,f_{6}=b_6',f_{7}=b_5',f_{8}=b_4', 
\end{gathered}
\end{equation}
besides we can exchange $a_1$ and $a_4$ for the choice of the component.\\
If we have $\gamma$ and two $\beta$ then we choose the basis in the following way
\begin{equation}
\begin{gathered}
    a_1 \rightarrow b_1 \rightarrow a_2 \rightarrow 0\\
    a_3 \rightarrow b_2 \rightarrow a_4 \rightarrow 0\\
    b_1' \rightarrow a_1' \rightarrow b_2' \rightarrow 0\\
    b_1'' \rightarrow a_1'' \rightarrow b_2'' \rightarrow 0\\
    e_{1}=a_1'+ia_1'',e_{2}=a_2, e_{3}=a_1, e_{4}=a_4,e_{5}=a_3,e_{6}=a_1'-ia_1'', f_{1}=b_2'+ib_2'', f_{2}=b_2'-ib_2'',\\ f_{3}=b_1, f_{4}=b_2,f_{5}=b_1'+ib_1'',f_{6}=b_1'-ib_1''
\end{gathered}
\end{equation}
For the choice of the component we can exchange $e_{1}$ with $e_{6}$, $f_{1}$ with $f_{2}$ and $f_{5}$ with $f_{6}$.\\
Finally, suppose we have $\delta$ and two $\alpha$. We will take $\delta_0$ as an example, because bigger cases do not add difficulty
\begin{equation}
\begin{gathered}
     b_1 \rightarrow  0\\
     b_2 \rightarrow  0\\
    a_1' \rightarrow b_1' \rightarrow a_2' \rightarrow b_2' \rightarrow a_3' \rightarrow 0\\
    a_1'' \rightarrow b_1'' \rightarrow a_2'' \rightarrow b_2'' \rightarrow a_3'' \rightarrow 0\\
    e_{1}=a_3',e_{2}=a_3'', e_{3}=a_2'+ia_2'', e_{4}=a_2'-ia_2'',e_{5}=a_1'',e_{6}=a_1',\\ f_{1}=b_1, f_{2}=b_2', f_{3}=b_2'', f_{4}=b_1'',f_{5}=b_1',f_{6}=b_2 
\end{gathered}
\end{equation}
And we can exchange $e_{3},e_{4}$.\\
\textbf{d)} In this case we are left with one block $\beta$ or block $\delta$ with block $\alpha$. However, the choice of basis is somewhat similar to the case a), so we skip the description.
\end{proof}

\begin{lem}
  \label{lem5}
  The map $pr_1$ is birational.
\end{lem}

\begin{proof} For the birationality it is sufficient to show that for a regular nilpotent $A$ (with orbit of maximal dimension that is) there exists the unique flag such that $A$ respects it. Suppose $m=2n+1$, a regular nilpotent has the type of $\alpha_n$ (remember that we can directly compute dimensions of the orbit and the null cone)
\begin{equation}
    a_1\rightarrow b_1 \rightarrow \dots \rightarrow b_{2n} \rightarrow a_{2n+1} \rightarrow 0
\end{equation}
Since we have conditions $AF_0^{(1)}\subset F_1^{(0)}=\{0\}$, then $F_0^{(1)} = \langle a_{2n+1} \rangle$. Now $A^*F_1^{(1)}\subset F_0^{(1)}$ therefore $F_1^{(1)}$ has to be equal to $\langle b_{2n}\rangle$. Obviously we can continue this kind of reasoning and the flag is uniquely determined.\\
Let $m=2n+2$ then a regular nilpotent has the form $\alpha_n + \alpha_0$
\begin{equation}
    a_1\rightarrow b_1 \rightarrow \dots \rightarrow b_{2n} \rightarrow a_{2n+1} \rightarrow 0,  a_{2n+2} \rightarrow 0
\end{equation}
In this case $F_0^{(1)}$ lies in the kernel of $A$, however the first half of a flag must be self-orthogonal and $a_{2n+2}$ can not participate in it. This is why a flag is uniquely determined until the vectors $a_{n+2}$ and $b_{n+1}$. Now the only vectors that we can put in the middle component of a flag $F_0^{(n+1)}$ are either $a_{n+1}+ia_{2n+2}$ or $a_{n+1}+ia_{2n+2}$, besides the dual vector to already chosen one will be added in the $F_0^{n+2}$ component (we have strengthened the remark above as we only need to know the first half of the self-orthogonal flag and the connectedness component because of self-orthogonality condition). So this choice depends only on the connectedness component of the flag variety $F_0$.\\
Let $m=2n$. A regular nilpotent has the type $\beta_n + \alpha_0$
\begin{equation}
    b_1 \rightarrow a_1 \rightarrow \dots \rightarrow a_{2n-1} \rightarrow b_{2n} \rightarrow 0, a_{2n} \rightarrow 0
\end{equation}
Here $a$ and $b$ change roles since $AF_0^{(i)} \subset F_1^{(i)},A^*F_1^{(i)} \subset F_0^{(i-1)}(i\le n)$. That is why $F_1^{(1)}=\langle b_{2n}\rangle, F_0^{(1)}=\langle a_{2n-1}\rangle$ (again $a_{2n}$ does not participate because of self-orthogonality) and so on. Again, we can add $a_{n-1}+ia_{2n}$ or $a_{n-1}-ia_{2n}$ in the middle component $F_0^{(n)}$ and this corresponds to the choice of the connectedness component.\\
Finally, let $m=2n-1$. A regular nilpotent has the type $\beta_n$
\begin{equation}
    b_1 \rightarrow a_1 \rightarrow \dots \rightarrow a_{2n-1} \rightarrow b_{2n} \rightarrow 0,
\end{equation}
thus we see that $F_{1}^{1}= \langle b_{2n}\rangle, F_{0}^{1}=\langle a_{2n-1} \rangle$ and so on.
\end{proof}

\begin{rem}
  The cases $m=2n-1,2n+2$ of~Lemma~\ref{lem5} were omitted in~\cite{GS}, and that is why their
  desingularization theorem was only proved for $m=2n$ and $m=2n+1$. We quote Remarque right
  before~\cite[Corollaire 2]{GS}: ``Dans les cas o\`u les rangs des deux composantes de la partie
  paire diff\`erent, cette
  construction ne permet pas de reconstituer un drapeau complet de $V_0\oplus V_1$ car tout
  \'el\'ement nilpotent de l’orbite maximale a un noyau de dimension au moins~2.''
\end{rem}

Combining~Proposition~\ref{lem4} and~Lemma~\ref{lem5} we obtain the following theorem.
\begin{thm}
  \label{thm1}
  The map $pr_1$ is a resolution of singularities in the cases $m=2n-1,2n,2n+\nobreak1,\allowbreak2n+\nobreak2$. \hfill $\Box$
  \end{thm}

    \subsection{Arbitrary $m,n$.}
The map $pr_1\colon  \Tilde{\mathcal{N}} \rightarrow \mathcal{N}$ can be defined for arbitrary
values of $m,n$ (see~Remarque right before~\cite[Th\'eor\`eme 1]{GS}): to this end we interpret
the fiber of the vector bundle $\Tilde{\mathcal{N}}$ over $(F_0^{(\bullet)},F_1^{(\bullet)})$ as
the odd part of the mixed Borel subalgebra compatible with
$(F_0^{(\bullet)},F_1^{(\bullet)})$~\cite[D\'efinition 5]{GS}.
One can ask whether this construction gives a resolution of singularities for general
values of $m,n$. Here we provide a counterexample of $m=2n-2$ and demonstrate that $pr_1$ is not a resolution of singularities in this case.\\
    It is easy to see that an element from the maximal orbit of $\mathcal{N}$ has the form $\beta_{n-1}+\beta_{1}$. 
    \begin{equation}
        b_1\rightarrow a_1 \rightarrow \dots \rightarrow a_{2n-3} \rightarrow b_{2n-2} \rightarrow 0, b_{2n-1}\rightarrow a_{2n-2} \rightarrow b_{2n}\rightarrow 0
    \end{equation}
    Now, if we start constructing $e_i',f_j$-basis as in \cite{GL}, we notice that the choice of order of $a_{n-1}\pm ia_{2n-2}$ vectors represents the choice of $\mathcal{B}_0$ component and we can choose arbitrary element $F_{1}^{(1)}=\langle \lambda b_{2n-2}+\mu b_{2n} \rangle, (\lambda:\mu)\in \mathbb{P}^1$ - in every case the operator $A$ preserves this flag. Therefore, $pr_1$ has fibers isomorphic to $\mathbb{P}^1$ over the open orbit in $\mathcal{N}$, so $pr_1$ can not be birational.\\
    Other cases are quite similar to the case above. Indeed, by dimension comparison via the formula provided in 7.1 here \cite{KP} (and by another useful formula in Corollary 6.1.4 here \cite{DC}) one can prove that in the case $m=2k+1<2n-1$ maximal nilpotent element has the type $\beta_{k+1}+(n-k-1)\delta_0$, in the case $m=2k<2n-1$ it has the type $\beta_k+ \beta_1+ (n-k-1)\delta_0$ and in the case $m>2n+2$ it has the type $\alpha_n+(m-2n-1)\alpha_1$. So we see that in all these cases we have the same ambiguity in the choice of $F_{1}^{(1)}$ ($m<2n-1$) or $F_{0}^{(1)}$ ($m>2n+2$). Therefore $pr_1$ can not be birational for $m$ not equal to $2n-1,2n,2n+1,2n+2$.\\
    Therefore in order to produce a resolution of singularities in the general case we need to alter our construction. We consider three general cases.\\
    
    \begin{defn}
  \label{defn2}
    \textbf{a)} If $m=2k<2n-1$ then we consider a closed subset (also a vector bundle) in $Hom(V_0,V_1)\times \mathcal{B}_0' \times \mathcal{B}_{1,k}$ cut out by equations
    \begin{equation}
        AF_{0}^{(i)} \subset F_1^{(i)}, A^* F_1^{(i)}\subset F_0^{(i-1)}, i\le k
    \end{equation}
    and dual relations obtained from these, where $\mathcal{B}_{1,k}$ is the variety of partial self-orthogonal flags $F_1^{(i)}, i\le k$ and $i\ge 2n-k$ in $V_1$. Note that the conditions above guarantee that $A$ is a nilpotent operator since operator $AA^*$ realizes full flag chain from $F_0^{(2k)}$ to $F_0^{(0)}$.\\
    
    \textbf{b)} If $m=2k+1<2n-1$ then we consider a closed subset in $Hom(V_0,V_1)\times \mathcal{B}_0' \times \mathcal{B}_{1,k+1}$ cut out by equations
    \begin{equation}
        AF_{0}^{(i)} \subset F_1^{(i)}, A^* F_1^{(i)}\subset F_0^{(i-1)}, i\le k+1
    \end{equation}
    and dual relations obtained from these. The conditions also guarantee that $A$ is nilpotent since we have a chain from $F_0^{(2k+1)}$ to $F_0^{(k)}$ and then we have a chain from $F_0^{(k+1)}$ to $F_0^{(0)}$.\\
    
    \textbf{c)} If $m> 2n+2$ then we consider a closed subset in $Hom(V_0,V_1)\times \mathcal{B}_{0,n} \times \mathcal{B}_{1}$ cut out by equations
    \begin{equation}
        AF_{0}^{(i)} \subset F_1^{(i-1)}, A^* F_1^{(i)}\subset F_0^{(i)}, i\le n
    \end{equation}
    and dual relations obtained from these, where $\mathcal{B}_{0,n}$ is the variety of partial self-orthogonal flags $F_0^{(i)}, i\le n$ and $i\ge m-n$ in $V_0$. Note that the conditions above guarantee that $A$ is a nilpotent operator since operator $A^*A$ realizes full flag chain from $F_1^{(2n)}$ to $F_1^{(0)}$. Additionally this time the variely $\mathcal{B}_{0,n}$ has only one component.
\end{defn}
    Then we have the following theorem.
    \begin{thm}
  \label{thm2}
  The new projection map to the first component  $pr_1'$ in the context of Definition~\ref{defn2} is a resolution of singularities of $\mathcal{N}$.
  \end{thm}
    \begin{proof}
    Following the proof of Theorem~\ref{thm1} we first check the uniqueness of flags in the preimage of a regular element in $\mathcal{N}$ under the map $pr_1'$.\\
    
    \textbf{a)} As we observed before a regular nilpotent element has the type $\beta_k+ \beta_1+ (n-k-1)\delta_0$. Then $F_1^{(1)}$ has to be spanned by a vector from $\ker A^*$ and lying in the image of $A$. The only possibility here is the rightmost vector from $\beta_k$ chain. Then it is easy to see that we can lift this chain uniquely up to $F_0^{(k-1)}$. And now in order to satisfy self-orthogonality condition we can add to $F_0^{(k)}$ either $a_{k}+ia_{2k}$ or $a_{k}-ia_{2k}$ (for the notation please see Proposition~\ref{lem4}). This choice represents the choice of component in $\mathcal{B}_0$. The rest of the flag is fixed by orthogonality conditions.\\
    
    \textbf{b)} In this case a regular nilpotent has the type $\beta_{k+1}+(n-k-1)\delta_0$. The $F_1^{(1)}$ is also fixed uniquely due to the same reasons as in the case a). The rest of the chain also lifts uniquely.\\
    
    \textbf{c)} In this case a regular nilpotent has the type $\alpha_n +(m-2n-1)\alpha_1$. This time we can fix uniquely $F_0^{(1)}$ and the rest of the chain lifts as in the previous cases.\\
    
    Therefore the map $pr_1'$ is birational in all these cases. We can also say that $pr_1'$ is surjective onto $\mathcal{N}$ since the vector bundles clearly map onto all regular elements in $\mathcal{N}$, but we know from \cite{GL} that the orbit of all regular nilpotent elements is open and dense in $\mathcal{N}$. We can clearly see that $pr_1'$ is a proper map, thus surjectivity follows. Unfortunately an explicit direct proof seems to be cumbersome, so we do not provide it here.
    \end{proof}
    
\section{Regular odd elements}
\label{sec3}

This section is devoted to description of regular elements in $\Hom(V_0,V_1)\cong V_0 \otimes V_1$.

\subsection{Weierstra\ss\ section}
As we already know, $\mathcal{N}$ contains an open orthosymplectic orbit, let us choose an element $u$ from it. 
\begin{equation}
    u=\sum_{i\le 2\lceil m/2 \rceil -1} e_{m+1-i}\otimes e_{i+1} (m\le 2n), u=\sum_{i\le 2n} e_{m+1-i}\otimes e_{i} (m\ge 2n+1)
\end{equation}
One can consider the normal bundle $L$ to $\mathcal{N}$ at $u$ which is identified with a vector subspace of $V_0 \otimes V_1$. We will prove the following proposition.\\
\begin{prop}
  \label{prop1}
  The variety $\Sigma :=u+L$ is a Weierstra\ss\ section.
  \end{prop}

\begin{proof}
  There exists a rational semisimple element $h\in \mathfrak{so}(V_0)\oplus \mathfrak{sp}(V_1)$ such that $hu=2u$. Indeed, we can find such an element by using the formulas for $u$ above.
\begin{equation}
    h=2 \diag(-2k,-2k+2,\dots , 2k) \oplus 2\diag(-2k-1,-2k+1,\dots, 2k+1,0,\dots, 0) (m=2k+1\le 2n)
\end{equation}
\begin{equation}
    h=2 \diag(0,-2k,-2k+2,\dots , 2k) \oplus 2\diag(-2k-1,-2k+1,\dots, 2k+1,0,\dots, 0) (m=2k+2\le 2n)
\end{equation}
\begin{equation}
    h=2 \diag(0,\dots,0,-2n,-2n+2,\dots , 2n) \oplus 2\diag(-2n+1,-2n+3,\dots, 2n-1) (m\ge 2n+1)
\end{equation}
We view $L$ and $T_u \mathcal{N}$ as an $h$-invariant subspace of $V_0\otimes V_1$, such that $L \oplus T_u \mathcal{N} = V_0 \otimes V_1$. Let us denote the eigenspaces of $h$ with eigenvalue $\lambda \in \mathbb{Q}$ as $V_{\lambda}(h)$. Note that the elements from $U:=\oplus_{\lambda>0} V_{\lambda}(h)$ lie in the null cone since they are retracted to $0$ by the action of $h$ and $u\in U$. Therefore $U\subset T_u \mathcal{N}$ and all eigenvalues of $h$ in $L$ are non-positive. In fact, these eigenvalues $c_i$ are equal to $2-2d_i$, where $d_i$ are degrees of generators of $\mathbb{C}[V_0\otimes V_1]^{SO(V_0)\times Sp(V_1)}$. Indeed, if we knew that the rank of $\pi\colon V_0 \otimes V_1 \rightarrow V_0 \otimes V_1/\!\!/SO(V_0)\times Sp(V_1)$ at the point $u$ is maximal, then $df_i$ (differentials of generators of the invariant algebra above) would be linearly independent at $u$ and would annihilate the $T_u\mathcal{N}$. The torus action of $\mathbb{C}^*$ on $V_0\otimes V_1$ as $t^{h-2}$ acts as identity on $u$ and as $t^{c_i-2}$ on the eigenspaces of $L$. On the other hand $\mathbb{C}^*$ acts as $t^{-2d_i}$ of $f_i$ (see Lemma~\ref{lem3}, these $f_i$ are pullbacks of generators of $\mathbb{C}[\mathfrak{so}(V_0)]^{SO(V_0)}$ or $\mathbb{C}[\mathfrak{sp}(V_1)]^{Sp(V_1)}$ under $q_0$ or $q_1$ respectively), from which we get $c_i=2-2d_i$. Since maximality of the rank $\pi$ is an open condition and we have an equivariant action of $\mathbb{C}^*$ on $u+L$ and on $V_0 \otimes V_1/\!\!/SO(V_0)\times Sp(V_1)$, then if we take a limit $t\rightarrow 0$, we will get that $\pi$ is dominant. The differential of $\pi$ is injective at $u$, therefore it is injective everywhere. Moreover, $\pi$ is injective, otherwise from the retracting action in any neighbourhood of $u$ we would have different points with the same value at $V_0 \otimes V_1/\!\!/SO(V_0)\times Sp(V_1)$, which is impossible. It means that the rank of $\pi$ at $u$ is maximal and $u+L$ is a Weierstrass section. The maximality of the rank follows from~\cite[Theorem 8.13]{PV}. All elements of $u+L$ have the orbit of maximal dimension by \cite[Proposition 8.12]{PV}: the values of $f_i$ at an orbit are fixed and $df_i$ are linearly independent due to
the retracting action.\end{proof}

\subsection{Regularity and self-supercommutators}
We define $(V_0\otimes V_1)^{\reg}=SO(V_0)\times Sp(V_1)(u+L)$. For $A\in V_0 \otimes V_1$ recall
$q_0(A)$ and $q_1(A)$ introduced in~Lemma~\ref{lem3}.
\begin{prop}
  \label{prop2}
\textup{a)} Let $m< 2n+1$, then 
\begin{equation}
    q_1(A)\in \mathfrak{sp}(V_1)^{\reg} \Rightarrow A \in (V_0\otimes V_1)^{\reg} \Rightarrow q_0(A)\in \mathfrak{so}(V_0)^{\reg}
\end{equation}
\textup{b)} Let $m \ge 2n+1$, then
\begin{equation}
    q_0(A)\in \mathfrak{so}(V_0)^{\reg} \Rightarrow A \in (V_0\otimes V_1)^{\reg} \Rightarrow q_1(A)\in \mathfrak{sp}(V_1)^{\reg}
\end{equation}
\end{prop}
Here $\mathfrak{sp}(V_1)^{\reg},\mathfrak{so}(V_0)^{\reg}$ are subsets of elements with orbits of maximal dimensions (i.e. $dim(\mathfrak{g})-rk(\mathfrak{g})$ for each Lie algebra $\mathfrak{g}$).
As we will see, we would need to prove the first two implications only in cases $m=2n-1,2n,2n+1,2n+2$, as otherwise images of maps $q_0,q_1$ into bigger Lie algebras do not contain any regular elements.
\begin{proof} It is easy to see that $q_0$ (resp.\ $q_1$) maps $u$ to a regular element (one can check that $q_1(u)$ is also regular when $m=2n+1$). Since regularity is an open condition, the neighbourhood of $u$ and in particular neighbourhood of $u$ in $\Sigma$ maps to a neighbourhood of the regular element.
  From the retracting equivariant action (dilation of the Lie algebra and conjugation by a group element does not change the size of stabilizer of any vector) and dominance of the map we deduce the last two implications.\\
Now, from the classification of A.~Berezhnoy~\cite{B} of all orbits in $V_0 \otimes V_1$ follows that $AA^*$ in the first case or $A^*A$ in the second case can be decomposed (in a basis from the classification) into images of $\alpha,\beta,\gamma,\delta,\epsilon$ blocks as in~\cite{KP} and direct sum with a nondegenerate parts of the following form
\begin{equation}
    \begin{pmatrix}
        \lambda^2 E_k + 2\lambda J_k & 0\\
        0 & -\lambda^2 E_k +2\lambda J_k
    \end{pmatrix},
\end{equation}
where $E_k$ is the identity matrix of size $k\times k$ and $J_k$ is the Jordan block of the same size. So the image of $V_0\otimes V_1$ under the map into the bigger algebra will contain regular elements only if $m=2n-1,2n,2n+1,2n+2$ (to see this one can compute Jordan normal form of the blocks above and notice that there is no way to construct Jordan normal form of a regular element in the bigger Lie algebra from these blocks). In order to show first two implications in cases $m=2n-1,2n,2n+1,2n+2$ we again need to invoke the Berezhnoy classification of orbits in the odd part of $\mathfrak{osp}(m|2n)$. Any regular element in $\mathfrak{so}(V_0)$ has nilpotent part of Jordan type $(2k+1)$ and a non-degenerate part, which size is maximal even and does not exceed $2\lfloor m/2 \rfloor$. Similarly, any regular element in $\mathfrak{sp}(V_1)$ has nilpotent part of Jordan type $(2k)$ and a non-degenerate part, which size is $2n-2k\le 2n$. It is also clear that the size of non-degenerate part of $A$ does not exceed $2n$ and $2\lfloor m/2 \rfloor$. Additionally, non-degenerate Jordan blocks of a regular element in $\mathfrak{so}(V_0),\mathfrak{sp}(V_1)$ can not have the same eigenvalues. From this description it is clear that the type of non-degenerate part of $A$ is fixed uniquely up to $SO(V_0)\times Sp(V_1)$ action. Since a regular element in $\mathfrak{so}(V_0)$ has the type $(2k+1)$ for some $k$ and in $\mathfrak{sp}(V_1)$  it has the type $(2k)$, the nilpotent part of $A$ must be either of type $\alpha$ or of type $\beta$ (and maybe plus some $\alpha_1$ or $\delta_0$ if they can be disregarded in the bigger Lie algebra). Therefore, in the cases $m=2n-1,2n$ the only possibility for nilpotent part of $A$ is to be of type $\beta_{k+1}$, in the case $m=2n+1$ it has to be $\alpha_k$ and in the case $m=2n+2$ it is also $\alpha_k$ as well.\\
\end{proof}

We can also show the following proposition.

\begin{prop}
  \label{prop3}
\textup{a)} Let $m=2n-1,2n,2n+1$, then 
\begin{equation}
    A \in (V_0\otimes V_1)^{\reg} \Rightarrow q_1(A)\in \mathfrak{sp}(V_1)^{\reg}
\end{equation}
\textup{b)} Let $m= 2n+1,2n+2$, then
\begin{equation}
    A \in (V_0\otimes V_1)^{\reg} \Rightarrow q_0(A)\in \mathfrak{so}(V_0)^{\reg}
\end{equation}
\end{prop}
\begin{proof}
The case of $m=2n+1$ is treated in~\ref{prop2}, but we will compute the $L$-space in this case anyway (because it might be useful). It is sufficient to show that in these cases the image of $\Sigma$ is contained in regular elements. For this we need to compute $L$.\\
Let $m=2n$, then w.l.o.g. $u=e_{2n-1}\otimes e_{1}+e_{2n-2}\otimes e_{2}+\dots + e_1 \otimes e_{2n-1}\in V_0 \otimes V_1$, where $e_i$ is a basis with an antidiagonal form in both spaces and $e_{2n}$ is orthonormal to the first $2n-1$ vectors in $V_0$. A small computation shows that $T_u \mathcal{N}=\mathfrak{so}(V_0)\oplus \mathfrak{sp}(V_1) u$ consists of the vectors like this $e_i\otimes e_j - e_j \otimes e_i (i,j <2n), e_{2n}\otimes e_i (i<2n),e_i \otimes e_1 (i<2n), e_{2n-j}\otimes e_{i}-e_{i-1}\otimes e_{2n+1-j} (1<i\le n,j\le n), e_{2n-j}\otimes e_{i}+e_{i-1}\otimes e_{2n+1-j} (i>n, j\le n \text{ or } 2n>j>n,i\le n)$. One can see from such description that $L$ has a basis $e_{2i}\otimes e_{2n}$ and from here
\begin{multline}
    (u+\sum_{i}\lambda_i e_{2i}\otimes e_{2n})(u+\sum_{i}\lambda_i e_{2i}\otimes e_{2n})^* = E_{12}+\dots +E_{n-2,n-1}-E_{n-1,n}-\dots E_{2n-1,2n}+\\+\sum_{i<n}\lambda_i (E_{2n,2n+1-2i}+E_{2i,1}) + \lambda_{n}^2 E_{2n,1},
\end{multline}
which is a regular element.\\
Let $m=2n-1$, then $u$ is defined by the same formula and by the action of the algebra on $u$ we are able to get $V_0 \otimes e_1, e_i\otimes e_j - e_{j}\otimes e_{i}(i,j<2n)$ and analogous to the case above expressons obtained from $\mathfrak{sp}(V_1)$. We can see that $L$ has the same basis $e_{2i}\otimes e_{2n}$, thus we have the same formulas with the regular elements.\\
Let $m=2n+1$, $u=e_{1}\otimes e_{2n+1}+\dots + e_{2n}\otimes e_{2} \in V_1 \otimes V_0$. From the action of $\mathfrak{sp}(V_1)$ we get $e_i \otimes e_{2n+2-j}\pm e_{2n+1-j}\otimes e_{i+1}$ with the same rules as above and $e_{2n+2-j}\otimes e_{i}-e_{i}\otimes e_{2n+2-j}$ from $\mathfrak{so}(V_0)$ (with exceptional situation when $j=1$ - the vectors will be $-e_i\otimes e_{2n+1}$). We can see that $e_i\otimes e_i, i\le n$ is a suitable basis for $L$, then after the projection of the slice we will get
\begin{equation}
    (u+\sum_{i\le n} \lambda_i e_i \otimes e_i)(u+\sum_{i\le n} \lambda_i e_i \otimes e_i)^*= E_{21}+\dots -E_{2n,2n-1}+\sum_{i\le n} 2\lambda_i E_{i,2n+1-i},
\end{equation}
in $\mathfrak{sp}(V_1)$, which is regular and
\begin{equation}
   (u+\sum_{i\le n} \lambda_i e_i \otimes e_i)^*(u+\sum_{i\le n} \lambda_i e_i \otimes e_i) =E_{2n+1,2n}+\dots - E_{21}+\sum_{i\le n}\lambda_i (E_{i,2n+3-i}-E_{i+1,2n+2-i})
\end{equation}
in $\mathfrak{so}(V_0)$, which is also regular.\\
If $m=2n+2$, $u$ is defined by the same formula and to the tangent space we also add $e_{2n+2-i}\otimes e_{2n+2},1<i<2n+2$, so $L$ consists of the same $e_i\otimes e_i, i\le n$ and projection to $\mathfrak{so}(V_0)$ will be the same, but it is still regular for $m=2n+2$.
\end{proof}

\section{Vanishing cohomology}
\label{sec4}
In this section we will use the results of~\cite{DP} in order to give a lower bound for cohomology vanishing of some vector bundles on our variety $\Tilde{\mathcal{N}}$.\\
Recall the context of~\cite[Theorem 3.1]{DP}: if for a semisinple Lie group $G$ we have a parabolic subgroup $P$ and a variety $Z=G \times_P N$, where $N$ is a $B$-submodule of a $G$-module $V$ normalized by $P$, then for a weight $\mu \in \mathfrak{X}^P$ from the character group of $P$ the dual sheaf of sections $\mathcal{L}_Z(\mu)^*$ of the line bundle $G\times_P (N \oplus \mathbb{C}_{\mu}) \rightarrow Z$ has the following property
\begin{equation}
    H^i(Z,\mathcal{L}_Z(\mu)^*)=0, i\ge 1 \text{ if } \mu \gtrdot |\Psi|-|\Delta(\mathfrak{n})|,
\end{equation}
provided that $Z\rightarrow G\cdot N$ is generically finite. Here $|\Psi|$ is the sum of weights in $N$, $|\Delta(\mathfrak{n})|$ is the sum of weights of nilpotent radical for $P$ and $\gtrdot$ is the dominance order.
\begin{prop}
  \label{prop4}
  Let $\epsilon_i$ (resp.\ $\delta_i$) denote the standard basis in the dual Cartan of
  $\mathfrak{so}(V_0)$ (resp.\ $\mathfrak{sp}(V_1)$). Then the lower bound for $\mu$ that
  guarantees the higher cohomology vanishing given
  by~\cite[Theorem 3.1]{DP} is equal to\\
a) $\sum_{i=1}^n (\epsilon_i-\delta_i)$ if $m=2n+1$.\\
b) $\sum_{i=1}^{n-1}\epsilon_i-\sum_{i=1}^n \delta_i$ if $m=2n-1$.\\
c) $0$ if $m=2n$ or $2n+2$.\\
d) $2(n-k)\sum_{i\le k} (\epsilon_i -\delta_i)$ if $m=2k<2n-1$.\\
e) $(2n-2k-1)(\sum_{i\le k}\epsilon_i - \sum_{i\le k+1}\delta_i)$ if $m=2k+1<2n-1$.\\
f) $(m-2n-2)\sum_{i\le n}(\delta_i-\epsilon_i)$ if $m > 2n+2$.
\end{prop}
\begin{proof}
  The only thing we need to notice is that the condition of generic finiteness holds true since $\Tilde{\mathcal{N}} \rightarrow \mathcal{N}$ is a resolution of singularities. The rest is a straightforward calculation.  
\end{proof}
\begin{rem}
    The positivity properties of orthosymplectic Kostka polynomials~\cite[\S3.3]{BFT} suggest
    a conjecture that $H^{>0}(\Tilde{\mathcal{N}},\mathcal{L}_Z(\mu)^*)=0$ for {\em any}
    dominant weight $\mu$.
  \end{rem}
\selectlanguage{english}

\end{document}